\documentclass[11pt]{amsart}

\pdfoutput=1

\usepackage[T1]{fontenc}
\usepackage{lmodern}

\usepackage[text={370pt,584pt},centering]{geometry}  

\usepackage{esint,amssymb}
\usepackage{graphicx}
\usepackage{subcaption}
\usepackage{MnSymbol}
\usepackage{mathtools}
\usepackage[colorlinks=true, pdfstartview=FitV, linkcolor=blue, citecolor=blue, urlcolor=blue,pagebackref=false]{hyperref}
\usepackage{microtype}

\usepackage{bm}
\usepackage{dsfont}
\usepackage{mathrsfs}
\usepackage{xcolor}
\usepackage{orcidlink}

\usepackage{caption}
\newtheorem{proposition}{Proposition}
\newtheorem{theorem}[proposition]{Theorem}

\theoremstyle{remark}
\newtheorem{remark}[proposition]{Remark}

\theoremstyle{definition}

\numberwithin{equation}{section}
\numberwithin{proposition}{section}

\newcommand{\R}{\mathbb{R}}

\newcommand{\ep}{\varepsilon}
\newcommand{\eps}{\varepsilon}

\renewcommand{\le}{\leqslant}
\renewcommand{\ge}{\geqslant}
\renewcommand{\leq}{\leqslant}

\renewcommand{\subset}{\subseteq}
\renewcommand{\bar}{\overline}

\newcommand{\td}{\widetilde}

\newcommand{\Ll}{\left}
\newcommand{\Rr}{\right}
\renewcommand{\d}{\mathrm{d}}
\newcommand{\dr}{\partial}

\newcommand{\mcl}{\mathcal}

\newcommand{\msc}{\mathscr}
\newcommand{\al}{\alpha}

\newcommand{\de}{\delta}

\DeclareMathOperator{\dist}{dist}

\newcommand{\E}{\mathbb E}
\renewcommand{\phi}{\varphi}

\begin{document}

\author[Jean-Christophe Mourrat]{Jean-Christophe Mourrat\,\orcidlink{0000-0002-2980-725X}}
\address[Jean-Christophe Mourrat]{Department of Mathematics, ENS Lyon and CNRS, Lyon, France}

\author{Loucas Pillaud-Vivien}
\address[Loucas Pillaud-Vivien]{CERMICS, CNRS, ENPC, Institut Polytechnique de Paris, Marne-la-Vallée, France}

\subjclass[2010]{}
\date{}

\title[Oscillating solutions to MFL-DA]{Oscillating solutions to the mean-field Langevin descent-ascent flow}

\begin{abstract}
We present a counterexample to the statement of convergence of the mean-field Langevin descent-ascent flow on $\R^2$. We consider payoff functions that are shaped as a double well in each coordinate, and for which the deterministic dynamics admits a limit cycle. When the coupling between the two coordinates is sufficiently strong and the entropic regularization sufficiently small, we show that the mean-field dynamics remains close to this cyclic behavior, and in particular, does not converge.
\end{abstract}

\maketitle

%
%
%
%
%
%

\section{Introduction}
\label{s.intro}

A natural procedure for finding a saddle point of a function $f :  \mathbb{R}^d \times \mathbb{R}^d \to \mathbb{R}$ is to iteratively take small steps in the direction of $(-\nabla_x f, \nabla_y f)$. In the limit of infinitesimally small steps, we obtain a continuous dynamics called the gradient descent-ascent flow. It is well known that this gradient descent-ascent flow can fail to converge, even for simple functions such as $f(x,y) = xy$, for which the flow circles around the origin. A natural idea to remedy this is to add entropic regularization, which amounts to optimizing over probability measures rather than points, and to study the resulting mean-field Langevin descent-ascent (MFL-DA) flow.
Given $\varepsilon > 0$ and a smooth function $f : \mathbb{R} \times \mathbb{R} \to \mathbb{R}$, the MFL-DA flow is given by
\begin{equation}
    \label{e.MFLDA}
        \begin{cases}
        \displaystyle{\dr_t \mu_t = \ep \, \dr_{x}^2 \mu_t + \dr_x \Ll( \mu_t \, \dr_x \int f(x,y) \, \d \nu_t(y) \Rr) , }
        \\
        \displaystyle{\dr_t \nu_t = \ep \, \dr_{y}^2 \nu_t -\dr_y \Ll( \nu_t \, \dr_y \int f(x,y) \, \d \mu_t(x) \Rr) .}
        \end{cases}
\end{equation}
This is the Wasserstein gradient descent-ascent flow for the functional
\begin{equation*}  
\mcl F(\mu, \nu) = \iint f \, \d \mu \, \d \nu + \ep \, H(\mu) - \ep \, H(\nu),
\end{equation*}
where $H(\mu) = \int \log \Ll( \frac{\d \mu}{\d x} \Rr) \, \d \mu$ is the (negative) differential entropy. The laws $(\mu_t, \nu_t)$ are the marginal laws of the processes $(X_t, Y_t)$ solving
\begin{equation}
\label{e.SDE.f}
\begin{cases}
\displaystyle{\d X_t = \Ll( -\int \dr_x f(X_t, y) \, \d \nu_t(y) \Rr) \d t + \sqrt{2\ep} \, \d B_t, } \\
\displaystyle{ \d Y_t = \Ll( \int \dr_y f(x, Y_t) \, \d \mu_t(x) \Rr) \d t + \sqrt{2\ep} \, \d B'_t,}
\end{cases}
\end{equation}
where $B_t, B'_t$ are independent standard Brownian motions. 

This framework was introduced and studied in \cite{domingo2020mean}, where it is shown that if the flow converges, then the limit must be the unique entropy-regularized mixed Nash equilibrium. Whether convergence actually holds was posed as an open problem in~\cite{wang2024open}. Part of the motivation for this question is that the entropic regularization endows the problem with strong structural properties: for instance, under mild assumptions, the partial equilibrium measures satisfy uniform log-Sobolev inequalities, which suggests that the regularization might tame the cycling behavior inherent to the deterministic dynamics.
Several partial results support this intuition. Local convergence has been established independently in~\cite{seo2026local} and~\cite{wang2026local}. Global convergence is known to hold under two-timescale modifications of the flow~\cite{lu2023two, ma2022provably}; see also \cite{kim2024symmetric} for other modified dynamics. However, in a companion paper~\cite{mourrat2026pl}, one of us showed that in the finite-dimensional setting, gradient descent-ascent can fail to converge even under a two-sided Polyak–Łojasiewicz condition — the natural finite-dimensional analog of the uniform log-Sobolev assumption~\cite{otto2000generalization}. The question of global convergence in the mean-field setting, however, remained open.

The point of this paper is to show that the answer is negative: there exist functions for which the MFL-DA flow stays in a compact region and yet does not converge. More precisely, we consider functions indexed by $\alpha \ge 1$ of the form: for every $x,y \in \R$,
\begin{equation}  
\label{e.def.f}
f_\al(x,y) := \al xy - \frac{x^2}{2} + \frac{x^4}{12} + \frac{y^2}{2} - \frac{y^4}{12}\, .
\end{equation}
For this choice of the function $f$, the well-posedness of the system \eqref{e.MFLDA} follows from \cite{dosreis2019freidlin, hong2024mckean}. As will be confirmed below (see Propositions~\ref{p.2moment.control} and \ref{p.moment.bounds}), the quartic large-scale behavior of $f$ ensures the confinement of the flow. 
Computing the partial derivatives: $\dr_x f_\al(x,y) = \al y - x + x^3/3,\, \dr_y f_\al(x,y) = \al x + y - y^3/3$ and writing $m_t := \int x \, \d \mu_t(x) = \E[X_t]$ and $n_t := \int y \, \d \nu_t(y) = \E[Y_t]$ for the means, the SDE system~\eqref{e.SDE.f} becomes
\begin{equation}
\label{e.SDEexplicit}
\begin{cases}
    \displaystyle{\d X_t = \Ll( -\al n_t + X_t - X_t^3/3 \Rr) \d t + \sqrt{2\ep} \, \d B_t, }\\
    \displaystyle{\d Y_t = \Ll( \al m_t + Y_t - Y_t^3/3 \Rr) \d t + \sqrt{2\ep} \, \d B'_t \, .}
\end{cases}
\end{equation}
For this choice, the deterministic dynamics ($\varepsilon = 0$, Dirac initial condition) admits a limit cycle $\msc C_\alpha$ that circles around the origin and is contained in the annulus
\begin{equation}
\label{e.def.mclA}
\mathcal{A} :=\{(m,n) \in \R^2 \ : \ 3 \le m^2 + n^2 \le 6\},
\end{equation}
see Proposition~\ref{p.limitcycle}. When the coupling parameter $\alpha$ is sufficiently large and the entropic regularization $\varepsilon$ is sufficiently small, we show that the stochastic mean-field dynamics remains close to this limit cycle for all times, and in particular does not converge. 
\begin{theorem}
\label{t.main}
For every $\de \in (0,1]$, $\al < +\infty$ sufficiently large, and $\ep > 0$ sufficiently small, the following holds. If $(X_0, Y_0)$ is deterministic and belongs to $\msc C_\al$, then for every $t_0 \ge 0$, there exists a parametrization $(\bar m_t, \bar n_t)_{t \in [t_0,t_0 + T_\al]}$ of the curve $\msc C_\al$ such that for every $t \in [t_0, t_0 + T_\al]$, we have
\begin{equation*}  
\E[|X_t - \bar m_t|^2] + \E[|Y_t - \bar n_t|^2] \le \delta.
\end{equation*}
In this circumstance, the means $(\E[X_t], \E[Y_t])$ therefore circle around the origin for all times while remaining away from it; in particular, the measures $(\mu_t, \nu_t)$ do not converge as $t$ tends to infinity.
\end{theorem}
%
We have stated Theorem~\ref{t.main} with the assumption that $(X_0,Y_0)$ is deterministic and belongs to $\msc C_\al$. This is only for the sake of clarity of exposition. The more general Theorem~\ref{t.bootstrap} below allows for random initial conditions, provided that they are sufficiently concentrated and with a mean sufficiently close to~$\msc C_\al$. Theorem~\ref{t.bootstrap} assumes a deterministic upper bound on $X_0^2 + Y_0^2$, but this itself can be weakened, as the proof will hopefully make clear.
%

The proof proceeds in two main steps. The first is a careful analysis of the deterministic dynamics, carried out in Section~\ref{s.limitcycle}, for which we establish the existence of a unique limit cycle via the Poincaré–Bendixson theorem and construct a Lyapunov function certifying its exponential stability (Propositions~\ref{p.limitcycle} and \ref{p.lyapunov}). The second step, in Section~\ref{s.toy}, is an inductive argument (Theorem~\ref{t.bootstrap}) showing that the stochastic fluctuations introduced by the noise remain small over each period of the cycle, so that they cannot accumulate and drive the system away from the periodic orbit. Finally, experiments illustrating the result are presented in Section~\ref{s.experiments}.

\section{Analysis of the deterministic dynamics}
\label{s.limitcycle}

In this section, we focus on the analysis of the deterministic dynamics, where we set $\ep = 0$ and we impose the initial measures $\mu_0$ and $\nu_0$ to be Dirac masses. In this case, the evolution \eqref{e.SDEexplicit} becomes deterministic and reduces to the differential system
\begin{equation}
\label{e.particleODE}
\dr_t m = -\al n + m - \frac{m^3}{3}, \qquad \dr_t n = \al m + n - \frac{n^3}{3}.
\end{equation}
We let $F_\al : \R^2 \to \R^2$ denote the vector field driving this system of equations, 
\begin{equation}  
\label{e.def.F}
F_\al (m,n) := \Ll(-\al n+m-\frac{m^3}{3}, \al m + n - \frac{n^3}{3}\Rr).
\end{equation}
We recall that the annulus $\mcl A$ is defined in \eqref{e.def.mclA}.
\begin{proposition}[limit cycle]
\label{p.limitcycle}
The system~\eqref{e.particleODE} admits a unique non-constant periodic orbit, which we denote by $\msc C_\al$. This orbit circles around the origin and is contained in $\mcl A$. It is a limit cycle of every trajectory of the system~\eqref{e.particleODE} except the one that stays put at the origin. The smallest period $T_\al$ of the motion along this cycle satisfies 
\begin{equation}  
\label{e.limitcycle.period}
\frac{4\pi}{2\al + 1} \le T_\al \le \frac{4\pi}{2\al - 1}.
\end{equation}
\end{proposition}

\begin{proof}
We set $r := \sqrt{m^2 + n^2}$. For $(m,n)$ solving \eqref{e.particleODE}, we can write
\begin{equation}  
\label{e.dr.r}
\dr_t (r^2)  = 2(m \, \dr_t{m} + n \, \dr_t{n}) = 2r^2 - \frac{2}{3}(m^4 + n^4).
\end{equation}
Since $\frac {r^4}{2} \le m^4 + n^4 \le r^4$, we have
\begin{equation*}  
2r^2  \Ll( 1 - \frac {r^2} 3 \Rr) \le \dr_t (r^2) \le r^2  \Ll( 2 - \frac {r^2} 3 \Rr) .
\end{equation*}
Hence, either $m(0) = n(0) = 0$ and the trajectory stays put at the origin, or $\liminf_{t \to +\infty} r^2 \ge 3$ and $\limsup_{t \to +\infty} r^2 \le 6$. Since the trajectory either stays put at the origin or never visits it, we can define an angular variable $\theta$ such that $m = r \cos \theta$ and $n = r \sin \theta$, and we have
\begin{equation}  
\label{e.drt.theta}
\dr_t \theta = \frac{m \dr_t n - n \dr_t m}{r^2} = \al + \frac{1}{3r^2} mn(m^2 - n^2) = \al + \frac{r^2}{12} \sin(4\theta),
\end{equation}
where we used $mn = \tfrac 1 2 r^2 \sin (2\theta)$ and $m^2 - n^2 =  r^2 \cos(2\theta)$ to derive the last identity. Since $\al \ge 1$ and $r^2 \le 6$ in the annulus $\mcl A$, we see that $\dr_t \theta$ never vanishes in $\mcl A$, and in particular the vector field $F_\al$ never vanishes in $\mcl A$. (We can also note that the system \eqref{e.particleODE} turns anticlockwise in $\mcl A$.) Using this, the previous observation on the liminf and limsup of $r^2$, and the Poincar\'e-Bendixson theorem \cite[Theorem 3.7.1]{perko2001differential}, we deduce that the system \eqref{e.particleODE} admits a limit cycle in~$\mcl A$. The divergence of the vector field $F_\al$ is $\nabla \cdot F_\al(m,n) = 2-r^2$, which is strictly negative everywhere in $\mcl A$. It thus follows from \cite[Theorem~3.4.2]{perko2001differential} that every limit cycle in $\mcl A$ is stable. If there were two distinct cycles in $\mcl A$, then there would need to be an unstable cycle in-between (since the Poincaré return map \cite[Section~3.4]{perko2001differential} cannot have two stable fixed points without having an unstable fixed point in-between), which we ruled out. Hence there is exactly one limit cycle for the system \eqref{e.particleODE}, which we denote by $\msc C_\al$. The fact that every trajectory except the one staying put at the origin is attracted to $\msc C_\al$ is again a consequence of the Poincar\'e-Bendixson theorem \cite[Theorem 3.7.1]{perko2001differential}. Recalling that $\dr_t \theta$ is bounded away from zero, we see that $\msc C_\al$ circles around the origin. It remains to see \eqref{e.limitcycle.period}. Let $(m_0, n_0)$ be a point on $\msc C_\al$, let $(m_t, n_t)_{t \ge 0}$ be the solution to \eqref{e.particleODE}, with $(\theta_t, r_t)_{t \ge 0}$ the associated polar coordinates, and let $T_\al > 0$ be the minimal period of this motion (recall that $F_\al$ never vanishes along $\msc C_\al$). By the Jordan curve theorem, the set $\R^2 \setminus \msc C_\al$ contains exactly two connected components. On the unbounded component, the winding number of $\msc C_\al$ is null. It jumps by $\pm 1$ unit upon crossing $\msc C_\al$, and the winding number of $\msc C_\al$ at the origin is $(\theta_{T_{\al}} - \theta_0)/(2\pi)$. By \eqref{e.drt.theta} and the fact that $\msc C_\al \subset \mcl A$, we have that $|\dr_t \theta - \al| \le \frac 1 2$, and since $\al \ge 1$, the winding number at the origin is $1$ and
\begin{equation*}  
T_\al \Ll( \al - \frac 1 2 \Rr) \le {\theta_{T_{\al}} - \theta_0} = 2\pi \le T_\al \Ll( \al + \frac 1 2 \Rr) .
\end{equation*}
This yields \eqref{e.limitcycle.period} and thus completes the proof.
\end{proof}
A key aspect we will exploit later is that $T_\al$ can be made arbitrarily small by sending $\al$ to infinity, while at the same time keeping the rough location of the cycle $\msc C_\al$ unchanged, in the sense that it always stays within the fixed annulus $\mcl A$. Although we will not need this, we mention that one can in fact describe the asymptotic shape of $\msc C_\al$ more precisely in the limit of large $\al$. Indeed, notice from \eqref{e.limitcycle.period} that $\dr_t (r^2)$ remains bounded uniformly over $\al \ge 1$ and $(m,n) \in \mcl A$. Hence, as $\al$ tends to infinity, the radius $r$ along the cycle~$\msc C_\al$ cannot change by more than $C/\al$ (since $T_\al \le 4\pi/\al$). In other words, the cycle~$\msc C_\al$ asymptotically becomes a circle as $\al$ tends to infinity. Moreover, we have that
\begin{equation*}  
0 =\int_0^{T_\al} \dr_t(r_t^2) \, \d t = \int_0^{T_\al} \Ll( 2 r_t^2 - \frac {2r^4} 3 (\cos^4 \theta_t + \sin^4 \theta_t)  \Rr) \, \d t .
\end{equation*}
Since for large $\alpha$, the radius $r$ is close to a constant $R_\al$ and $\dr_t \theta$ is close to the constant $\al$, and using also that $\cos^4$ and $\sin^4$ average to $\frac 3 8$, we get from the previous display that $R_\al^2 \simeq R_\al^4/4$, that is, $R_\al \simeq 2$. In other words, the cycle $\msc C_\al$ tends to the circle of radius $2$ centered at the origin as $\al$ tends to infinity.

Using symmetry, we record information on the first and second moments of $m$ and $n$ along the cycle $\msc C_\al$. 
\begin{proposition}[square average]
\label{p.square} Consider the system \eqref{e.particleODE} initialized at a point in the cycle $\msc C_\al$, and recall that we denote by $T_\al$ the minimal period of the resulting trajectory. We have
\begin{equation}
\label{e.mean}
\frac 1 {T_\al} \int_0^{T_\al} m_t \, \d t = \frac 1 {T_\al} \int_0^{T_\al} n_t \, \d t = 0
\end{equation}
and
\begin{equation}  
\label{e.square}
\frac 1 {T_\al} \int_0^{T_\al} m_t^2 \, \d t = \frac 1 {T_\al} \int_0^{T_\al} n_t^2 \, \d t \ge \frac 3 2.
\end{equation}
\end{proposition}
\begin{proof}
We first observe that the function $t \mapsto (-n_t, m_t)$ is also a solution to~\eqref{e.particleODE}. By uniqueness of the cycle $\msc C_\al$, we deduce \eqref{e.mean} and the equality in~\eqref{e.square}. Since the cycle $\msc C_\al$ is contained in the annulus $\mcl A$, we have $m^2 + n^2 \ge 3$ everywhere in $\msc C_\al$, and this completes the proof of the proposition.
\end{proof}
\begin{remark}  
As per our previous discussion on the fact that $\msc C_\al$ tends to the circle of radius $2$ as $\al$ tends to infinity, one could show that $\frac 1 {T_\al} \int_0^{T_\al} m_t^2 \, \d t$ and $\frac 1 {T_\al} \int_0^{T_\al} n_t^2 \, \d t$ tend to $2$ as $\al$ tends to infinity.
\end{remark}

We can strengthen the statement of stability of the cycle $\msc C_\al$ as follows. Recall the notation for the vector field $F_\al$ in \eqref{e.def.F}. 
\begin{proposition}[Lyapunov function]
\label{p.lyapunov}
Let $g(z) := \dist^2(z,\msc C_\al)$ denote the squared distance between $z \in \R^2$ and the limit cycle $\msc C_\al$. There exist $\delta, c > 0$ such that for every $\al$ sufficiently large and $z \in \R^2$, we have
\begin{equation*}  
g(z) \le \de  \quad \implies \quad (F_\al \cdot \nabla g)(z) \le -c  g(z).
\end{equation*}
\end{proposition}
\begin{proof}
Since $\msc C_\al$ is a smooth simple closed curve, we can find a neighborhood $U$ of $\msc C_\al$ in which the nearest-point projection
$\pi : U \to \msc C_\al$ is well-defined and smooth, so that for every $z \in U$, we have
\begin{equation*}  
g(z) = \inf_{z' \in \msc C_\al} |z-z'|^2 = |z - \pi(z)|^2.
\end{equation*} 
The optimality condition yields that $z-\pi(z)$ is orthogonal to the tangent to~$\msc C_\al$ at $\pi(z)$, which is oriented along $F_\al(\pi(z))$, that is, we have 
\begin{equation}
\label{e.ortho.proj}
(z-\pi(z)) \cdot F_\al(\pi(z)) = 0.
\end{equation}
By the envelope theorem \cite[Theorem~2.21]{HJbook}, we also see that 
\begin{equation}\label{e.gradg}
\nabla g(z) = 2\Ll(z - \pi(z)\Rr),
\end{equation}
and thus
\begin{equation*}  
F_\al \cdot \nabla g(z) = 2\Ll(z - \pi(z)\Rr) \cdot F_\al(z).
\end{equation*}
Using \eqref{e.ortho.proj}, we can rewrite this as
\begin{align*}  
F_\al \cdot \nabla g(z)  
& = 2\Ll(z - \pi(z)\Rr) \cdot \Ll(F_\al(z) - F_\al(\pi(z))\Rr) \\
& = 2\Ll(z - \pi(z)\Rr) \cdot \nabla F_\al(\pi(z))\Ll(z - \pi(z)\Rr) + O\Ll(|z - \pi(z)|^3\Rr),
\end{align*}
where $\nabla F_\al(u)$ denotes the Jacobian of $F_\al$ at $u \in \R^2$, and where we used that the second derivatives of $F_\al$ do not depend on $\al$. Denoting $\nu(z) := \frac{z - \pi(z)}{|z - \pi(z)|}$ for a unit normal to $\msc C_\al$ at $\pi(z)$, we can rewrite this as
\begin{equation}
\label{e.Lg.expansion}
F_\al \cdot \nabla g(z) = g(z)\Ll[2 \, \nu(z) \cdot \nabla F_\al(\pi(z))\nu(z) + O\Ll(\sqrt{g(z)}\Rr)\Rr].
\end{equation}
Recalling that $\nu(z)$ is orthogonal to $F_\al(\pi(z)) = (F_{\al,1}(\pi(z)), F_{\al,2}(\pi(z)))$, and that the choice of orientation of $\nu(z)$ is immaterial, we can take $\nu(z) = (F_{\al,2}(\pi(z)), -F_{\al,1}(\pi(z)))/|F_\al|$ (recall from Proposition~\ref{p.limitcycle} that $F_\al$ only vanishes at the origin). In view of \eqref{e.Lg.expansion}, the proof of Proposition~\ref{p.lyapunov} will be complete provided that we can show that, for $\al$ sufficiently large,
\begin{equation}
\label{e.contract.bound}
\sup_{z \in \mcl A}
\begin{pmatrix}  
F_{\al,2}(z) \\
-F_{\al,1}(z)
\end{pmatrix}
\cdot \nabla F_\al(z)
\begin{pmatrix}  
F_{\al,2}(z) \\
-F_{\al,1}(z)
\end{pmatrix} < 0.
\end{equation}
We have
\[
\nabla F_\al(m,n) = 
\begin{pmatrix} 
1 - m^2 & -\al \\ 
\al & 1 - n^2 
\end{pmatrix}.
\]
Noting that the statement \eqref{e.contract.bound} only concerns the symmetric part of $\nabla F_\al$, we can rewrite it as
\begin{equation}  
\label{e.inf.pos}
\inf \Ll\{ (m^2 - 1) F_{\al,2}(m,n)^2 + (n^2 - 1) F_{\al,1}(m,n)^2  \ : \ m^2 + n^2 \in [3,6]\Rr\}  > 0.
\end{equation}
We now proceed to prove \eqref{e.inf.pos}.
We have
\begin{multline}
\label{e.expand.alpha}
(m^2 - 1) F_{\al,2}(m,n)^2 + (n^2 - 1) F_{\al,1}(m,n)^2
\\
= \al^2 (m^4 + n^4 - m^2 - n^2) + R_\al(m,n),
\end{multline}
where
\begin{equation*}
R_\al(m,n) := \frac{4\al}{3} mn(m^2-n^2) + (m^2-1)\Ll(n - \frac{n^3}{3}\Rr)^2 + (n^2-1)\Ll(m - \frac{m^3}{3}\Rr)^2.
\end{equation*}
We can find a constant $C < +\infty$ such that for every $(m,n)$ with $m^2 + n^2 \le 6$, we have $|R_\al(m,n)| \le C\al$. 
We now bound the leading-order term from below. Since $m^4 + n^4 \ge \frac 1 2 (m^2 + n^2)^2$, we have that
\begin{equation*}  
m^4 + n^4 - m^2 - n^2 \ge \frac 1 2 r^4 - r^2 = \frac 1 2 r^2(r^2-2). 
\end{equation*}
Since $r^2 \ge 3$ in the annulus $\mcl A$, this quantity is bounded from below by $3/2$. Returning to \eqref{e.expand.alpha}, we thus obtain that for every $(m,n) \in \mcl A$,
\begin{equation*}
(m^2 - 1) F_{\al,2}(m,n)^2 + (n^2 - 1) F_{\al,1}(m,n)^2 \ge \frac{3\al^2}{2} - C\al.
\end{equation*}
For $\al$ sufficiently large, the right-hand side is strictly positive. This completes the proof of \eqref{e.inf.pos}, and therefore of the proposition. 
\end{proof}

%
%
%
%
%
%

\section{Mean-field dynamics}
\label{s.toy}

We now come back to the study of the mean-field SDE in \eqref{e.SDEexplicit}, with the notation $m_t = \E[X_t]$ and $n_t = \E[Y_t]$; in other words we now consider the case of $\ep > 0$. An important ingredient of our analysis, extending the fact that the deterministic dynamics studied in the previous section remains in a fixed bounded set uniformly over $\al \ge 1$, is that $(X_t, Y_t)$ satisfy moment bounds that are uniform in $\al$. We start with estimates on the second moments.
\begin{proposition}[Second moment bound]
\label{p.2moment.control}
For every $\alpha \ge 0$, $\ep \in (0,1]$ and $t \ge 0$, we have
\begin{equation}  
\label{e.2moment.control}
\E[X_t^2 + Y_t^2] \le \max \Ll( 8, \E[X_0^2 + Y_0^2] \Rr) .
\end{equation}
In particular, if $\E[X_0^2 + Y_0^2] \le 8$, then $|m_t|^2 + |n_t|^2 \le 8$ for every $t \ge 0$.
\end{proposition}
\begin{proof}
By It\^o's formula, we have
\begin{equation*}  
\frac 1 2 \dr_t \E[X_t^2] = -\alpha m_t n_t + \E[X_t^2] - \frac {\E[X_t^4]}{3} + \ep,
\end{equation*}
\begin{equation*}  
\frac 1 2 \dr_t \E[Y_t^2] = \alpha m_t n_t + \E[Y_t^2] - \frac {\E[Y_t^4]}{3} + \ep.
\end{equation*}
Denoting $\phi(t) := \E[X_t^2] + \E[Y_t^2]$, and using that
\begin{equation*}  
\E[X_t^4] + \E[Y_t^4] \ge \E[X_t^2]^2 + \E[Y_t^2]^2 \ge \frac 1 2 \phi(t)^2,
\end{equation*}
we obtain that 
\begin{equation*}  
\frac 1 2 \dr_t \phi(t) \le \phi(t) - \frac 1 6 \phi(t)^2 + 2\ep.
\end{equation*}
The quantity on the right side is negative whenever $\phi(t) \ge 8$, so \eqref{e.2moment.control} follows by a comparison argument. The last part of the proposition statement is obtained using the Cauchy-Schwarz inequality.
\end{proof}
We define
\begin{equation}
\label{e.def.tdX.tdY}
\td X_t := X_t - m_t, \qquad \td Y_t := Y_t - n_t.
\end{equation}
Noting that 
\begin{equation}  
\label{e.ev.mn}
\dr_t m_t = -\alpha n_t + m_t - \frac{\E[X_t^3]}{3}, \qquad \dr_t n_t = \alpha m_t + n_t - \frac{\E[Y_t^3]}{3},
\end{equation}
we obtain that
\begin{equation}
\label{e.SDE.for.td}
\d \td X_t = \Ll[(1-m_t^2)\td X_t - m_t \Ll( \td X_t^2 - \E[\td X_t^2] \Rr) - \frac 1 3 \Ll( \td X_t^3 - \E[\td X_t^3] \Rr) \Rr]  \d t + \sqrt{2\ep} \, \d B_t,
\end{equation}
and similarly for $\td Y$. 
\begin{proposition}[Moment bounds]
\label{p.moment.bounds} 
For every even integer $k \ge 2$ and provided that $\E[X_0^k + Y_0^k]$ is finite, there exists a constant $C_k$ depending only on $k$ and $\E[X_0^k + Y_0^k]$ such that for every $\al \ge 0$, $\ep \in (0,1]$, and $t \ge 0$, we have
\begin{equation*}  
\E[\td X_t^k + \td Y_t^k] \le C_k.
\end{equation*}
\end{proposition}
\begin{proof}
By It\^o's formula, we have for every integer $k \ge 2$ that
\begin{multline}  
\label{e.ito.pmoment}
\frac 1 k \dr_t \E[\td X_t^k] 
 = \E \Ll[\td X_t^{k-1}\Ll( (1-m_t^2) \td X_t -m_t \Ll(\td X_t^2 - \E[\td X_t^2] \Rr)
 - \frac 1 3 \Ll( \td X_t^3 - \E[\td X_t^3] \Rr)\Rr) \Rr] 
 \\ 
 + \ep (k-1) \E[\td X_t^{k-2}].
\end{multline}
We start with the analysis of the case $k = 4$. Noting that $\frac 1 3 = \frac{1/2}{2} + \frac{1/2}{6}$, we have by H\"older's inequality that
\begin{equation}  
\label{e.holder.326}
\E[|\td X_t|^3] \le \E[\td X_t^2]^{\frac 3 4} \E[\td X_t^6]^{\frac 1 4}.
\end{equation}
By Proposition~\ref{p.2moment.control}, we have that 
\begin{equation*}  
\E[\td X_t^2] \le \E[X_t^2] \le C, \qquad |m_t|^2 + |n_t|^2 \le C,
\end{equation*}
where the constant $C < +\infty$ is allowed to depend on an upper bound on $\E[X_0^2+ Y_0^2]$. We thus find that, possibly after enlarging the constant $C$, 
\begin{equation*}  
\frac 1 4 \dr_t \E[\td X_t^4] \le \E[\td X_t^4] + C \Ll( \E[|\td X_t|^5] +  \E[|\td X_t|^3] + \E[\td X_t^6]^\frac 1 2 + \ep \Rr) - \frac 1 3 \E[\td X_t^6] .
\end{equation*}
For some constant $C' < +\infty$, we have for every $x \in \R$ that $C |x|^5 \le \frac{x^6}{9} + C'$ and $C \sqrt{|x|} \le \frac {|x|}{9}+ C'$, hence 
\begin{equation*}  
\frac 1 4 \dr_t \E[\td X_t^4] \le \E[\td X_t^4] + C \E[|\td X_t|^3] - \frac 1 9 \E[\td X_t^6] + 2CC' +  C \ep.
\end{equation*}
Denoting $\phi(t) := \E[\td X_t^4]$, using Jensen's inequality and the assumption $\ep \le 1$, and reabsorbing everything into a larger constant $C < +\infty$, we have that
\begin{equation*}  
\frac 1 4 \dr_t \phi(t) \le \phi(t) + C \phi(t)^{\frac 3 4} - \frac 1 9 \phi(t)^{\frac 3 2} + C.
\end{equation*}
The right-hand side is negative whenever $\phi(t)$ is sufficiently large, so we obtain the result for the fourth moment of $\td X_t$. 

Now that $|\E[\td X_t^3]|$ is controlled, a similar argument allows one to bound the moments of $\td X_t$ order $k$ for every even integer~$k$ from \eqref{e.ito.pmoment}. The argument for controlling the moments of $\td Y_t$ is identical.
\end{proof}

For convenience, we use the notation $p_t := (m_t, n_t)$. We also recall the notation $T_\al$ from Proposition~\ref{p.limitcycle} for the minimal period of the deterministic dynamics \eqref{e.particleODE} along the cycle. Our main result Theorem~\ref{t.main} is a direct consequence of the following inductive structure.
\begin{theorem}[Induction]
\label{t.bootstrap}
There exists a constant $C < +\infty$ such that the following holds for every $k  > \frac 3 2$, $\alpha < +\infty$ sufficiently large, $\eps > 0$ sufficiently small, initialization $X_0, Y_0$ with $X_0^2 + Y_0^2 \le 8$, and $t_0 \ge 0$. If
\begin{equation}  
\label{e.bootstrap.ass}
\E[\td X_{t_0}^2]\le \al^{-k}, \quad \E[\td Y_{t_0}^2]\le \al^{-k}, \quad \dist(p_{t_0}, \msc C_\al) \le  \alpha^{1 - k},
\end{equation}
then the same inequalities are valid at time $t_0 + T_\al$, that is,
\begin{equation}  
\label{e.bootstrap.ccl}
\E[\td X_{t_0+T_\al}^2]\le \al^{-k}, \quad \E[\td Y_{t_0+T_\al}^2]\le \al^{-k}, \quad \dist(p_{t_0+T_\al}, \msc C_\al) \le \alpha^{1 - k},
\end{equation}
and moreover, letting $\bar p_{t_0}$ denote the point on $\msc C_\al$ that is closest to $p_{t_0}$ and $(\bar p_t)_{t \ge t_0}$ denote the solution to the deterministic system in \eqref{e.particleODE}, we have for every $t \in [t_0,t_0+T_\al]$ that
\begin{equation}  
\label{e.bootstrap.cycle}
|p_t - \bar p_t| \le C \al^{1 -k}.
\end{equation}
\end{theorem}
\begin{proof}
We fix $k > \frac 3 2$. Under our assumption that the initial condition must satisfy $X_0^2 + Y_0^2 \le 8$, we can bound all moments of $X_t$ and $Y_t$ uniformly over $\alpha \ge 0$, $\ep \in (0,1]$, and $t \ge 0$. In particular, we have $|p_t|^2 = |m_t|^2 + |n_t|^2 \le 8$ for all $t \ge 0$. From now on, we assume that \eqref{e.bootstrap.ass} holds. 

\medskip

\noindent \emph{Step 1.} In this step, we show that for $\alpha$ sufficiently large and $\eps > 0$ sufficiently small, we have for every $t \in [t_0, t_0 + T_\al]$ that $\E[\td X_t^2] \le 2\al^{-k}$ and $\E[\td Y_t^2] \le 2\al^{-k}$. Using that $|m_t|^2 \le 8$ and the It\^o identity \eqref{e.ito.pmoment}, we have that
\begin{equation*}  
\frac 1 2 \dr_t \E[\td X_t^2] \le  \E[\td X_t^2] + 3 \E[|\td X_t|^3] + \ep.
\end{equation*}
Since by Proposition~\ref{p.moment.bounds} all moments of $\td X_t$ are bounded, we can appeal to H\"older's inequality to find a constant $C < +\infty$ such that $\E[|\td X_t|^3] \le C \E[\td X_t^2]^{1- \frac 1 {2k}}$, so that
\begin{equation}  
\label{e.drX2.crude.bound}
\frac 1 2 \dr_t \E[\td X_t^2] \le  \E[\td X_t^2] + C \E[\td X_t^2]^{1-\frac 1 {2k}} + \ep.
\end{equation}
Denoting $T := \sup \Ll\{ t \in [t_0, t_0 + T_\al] \ : \  \E[\td X_t^2] \le 2 \al^{-k} \Rr\}$, using that $T_\al \le 4\pi/\al$ and redefining $C < +\infty$ as necessary, we have for every $t \in [t_0, T]$ that 
\begin{equation*}  
\E[\td X_t^2] \le \al^{-k} + C \al^{-1} \Ll( \al^{-k} + \al^{\frac 1 2 - k} + \ep \Rr)  .
\end{equation*}
For $\alpha$ sufficiently large and $\ep > 0$ sufficiently small, we can make sure that the right-hand side above is smaller than $2\al^{-k}$, and therefore that $T = t_0 + T_\alpha$ as desired. The argument for $\E[\td Y_t^2]$ is identical. 

\medskip

\noindent \emph{Step 2.} In this step, we show that provided that $\alpha$ is sufficiently large, we have for every $t \in [t_0, t_0 + T_\al]$ that 
\begin{equation}  
\label{e.dist.unif.time}
\dist(p_t, \msc C_\al) \le \al^{1-k}.
\end{equation}
Recall the notation $F_\al$ in~\eqref{e.def.F}, and the identity \eqref{e.ev.mn} which we can rewrite as
\begin{equation}  
\label{e.evol.p.R}
\dr_t p_t = F_\al(p_t) + R_t, 
\end{equation}
where we have set
\begin{align*}  
R_t & := \frac 1 3 (m_t^3 - \E[X_t^3], n_t^3 - \E[Y_t^3]) 
\\
& =  \Ll(-m_t \E[\td X_t^2] -\frac 1 3 \E[\td X_t^3], -n_t \E[\td Y_t^2] - \frac 1 3 \E[\td Y_t^3]\Rr).
\end{align*}
Recalling that $\td X_t$ has bounded moments of every order, and using the result of the previous step, we can argue as in the sentence above \eqref{e.drX2.crude.bound} to find a constant $C < +\infty$ such that for every $t \in [t_0, t_0 + T_\al]$,
\begin{equation}  
\label{e.bound.R}
|R_t| \le C \al^{\frac 1 2 - k}.
\end{equation}
Recall that we write $g(z) = \dist^2(z, \msc C_\al)$. We have
\begin{equation*}  
\dr_t (g(p_t)) = (F_\al \cdot \nabla g)(p_t) + R_t \cdot \nabla g(p_t). 
\end{equation*}
Let $\delta, c > 0$ be as given by Proposition~\ref{p.lyapunov}, and let $T := \sup \{ t \in [t_0, t_0 + T_\al] \ : \ g(p_t) \le \de\}$. For every $t \in [t_0, T]$, we can appeal to this proposition to bound the first term in the previous display by $-c g(p_t)$. 
Using that $|\nabla g(z)| = 2  \sqrt{g(z)}$, the Cauchy-Schwarz and Young's inequalities, we have
\begin{equation*}  
|R_t \cdot \nabla g(p_t)| \le 2 \sqrt{g(p_t)} |R_t| \le \frac{c}{2} g(p_t) + \frac {2|R_t|^2}{c}.
\end{equation*}
Combining these and enlarging the constant $C < +\infty$, we obtain that for every $t \in [t_0, T]$,
\begin{equation*}  
\dr_t (g(p_t)) \le -\frac c 2 g(p_t) + C \al^{1-2k}.
\end{equation*}
Integrating this (and updating $C$) yields that for every $t \le T$, 
\begin{equation}
\notag
g(p_t)  \le e^{-\frac{c}{2}t} \al^{2-2k} + C \al^{1-2k} (1 - e^{-\frac c 2 t}).
\end{equation}
By taking $\al \ge C$, we ensure that the right-hand side of the inequality above remains below $\al^{2-2k}$. In particular, since $k > 1$, we have $T = t_0 + T_\al$ and the announced inequality \eqref{e.dist.unif.time} is proved.

\medskip

\noindent \emph{Step 3.} For $\alpha$ sufficiently large, the point in $\msc C_\al$ that is closest to $p_{t_0}$ is well-defined; we denote it by $\bar p_{t_0} = (\bar m_{t_0}, \bar n_{t_0})$, and let $\bar p_t = (\bar m_t, \bar n_t)_{t \ge t_0}$ denote the solution to the system \eqref{e.particleODE} with this initialization. In this step, we show that for a sufficiently large constant $C < +\infty$, the inequality \eqref{e.bootstrap.cycle} is valid. We denote the $L^\infty$ norm of $\nabla F_\al$ on the ball of radius $\sqrt{8}$ by $L_\al$. Recalling that $(p_t)_{t \ge t_0}$ and $(\bar p_t)_{t \ge t_0}$ both take values in this ball, as well as \eqref{e.evol.p.R} and \eqref{e.bound.R}, we have for almost every $t \in [t_0, t_0 + T_\al]$ that
\begin{equation*}  
\dr_t \Ll(|p_t - \bar p_t|\Rr) \le L_\al |p_t - \bar p_t| + C \al^{\frac 1 2 - k},
\end{equation*}
with
\begin{equation*}  
|p_0 -\bar p_0| \le \al^{1-k}.
\end{equation*}
Integrating this, we find that for every $t \in [t_0, t_0 + T_\al]$,
\begin{equation*}  
|p_t -\bar p_t| \le e^{L_\al T_\al} \Ll(\al^{1-k} + C \al^{\frac 1 2 - k}\Rr) .
\end{equation*}
The conclusion follows by noting that $L_\al \le C \al$ while $T_\al \le 4\pi/\al$.

\medskip

\noindent \emph{Step 4.} We now show that $\E[\td X_{t_0 + T_\al}^2] \le \al^{-k}$ and $\E[\td Y_{t_0 + T_\al}^2] \le \al^{-k}$. As preparation for this, we deduce from the previous step and Proposition~\ref{p.square} that
\begin{equation}  
\label{e.nice.m}
\Ll| \frac 1 {T_\al} \int_{t_0}^{t_0 + T_\al} m_t \, \d t \Rr| \le C \al^{1-k}, \qquad \frac 1 {T_\al} \int_{t_0}^{t_0 + T_\al} m_t^2 \, \d t \ge \frac 5 4.
\end{equation}
Using \eqref{e.ito.pmoment} with $k = 2$, we have
\begin{equation*}  
\frac 1 2 \dr_t \E[\td X_t^2] = (1-m_t^2) \E[\td X_t^2] - m_t \E[\td X_t^3] - \frac 1 3 \E[\td X_t^4] + \ep.
\end{equation*}
We recall from \eqref{e.drX2.crude.bound} that $\dr_t \E[\td X_t^2] \le C \al^{\frac 1 2 - k} + \ep$, and a small modification of the argument leading to this bound, starting from the previous display, allows one to bound the absolute value of $\dr_t \E[\td X_t^2]$ by the same expression. For $\ep > 0$ sufficiently small, using again that $T_\al \le 4\pi/\al$, we thus have that
\begin{equation*}  
\int_{t_0}^{t_0 + T_\al}(1-m_t^2)\Ll( \E[\td X_t^2] - \E[\td X_{t_0}^2]\Rr)\, \d t \le C \al^{-\frac 3 2 - k}.
\end{equation*}
Using \eqref{e.ito.pmoment} with $k = 3$ and similar arguments as those leading to the bound on $|\dr_t \E[\td X_t^2]|$, one can show that $|\dr_t \E[\td X_t^3]| \le C \al^{\frac 1 2 - k}$, and therefore
\begin{equation*}  
\int_{t_0}^{t_0 + T_\al} m_t \Ll( \E[\td X_{t_0}^3] - \E[\td X_{t}^3] \Rr) \, \d t \le C \al^{-\frac 3 2 - k}.
\end{equation*}
Combining the last three displays (and using again that we take $\ep > 0$ sufficiently small and adjusting $C$), we obtain that 
\begin{multline*}  
\E[\td X_{t_0 + T_\al}^2] 
\le \E[\td X_{t_0}^2]\Ll(1 + 2 \int_{t_0}^{t_0 + T_\al} (1-m_t^2)\, \d t \Rr) 
\\
- 2 \E[\td X_{t_0}^3] \int_{t_0}^{t_0 + T_\al} m_t \, \d t + C \al^{-\frac 3 2 - k}.
\end{multline*}
Using again that $|\E[\td X_{t_0}^3]| \le C \al^{\frac 1 2 - k}$ and \eqref{e.nice.m} yields that
\begin{equation*}  
\E[\td X_{t_0 + T_\al}^2] \le \E[\td X_{t_0}^2] \Ll( 1 - \frac{T_\al}{2} \Rr) + C \al^{\frac 1 2 - 2k} + C \al^{-\frac 3 2 - k}.
\end{equation*}
Recalling that $\E[\td X_{t_0}^2] \le \al^{-k}$, that $T_\al \ge 2\pi/\al$, and that $k > \frac 3 2$, we obtain the desired result by taking $\al$ sufficiently large.
\end{proof}

\section{Numerical experiments}
\label{s.experiments}

In this section, we display some numerical experiments showing the behavior described in the main theorems. For this we compute a particle approximation of the minimax SDE that we study, that is, we take particles $(X_i, Y_i)_{1 \leq i \leq n}$ and solve
\begin{align}
\label{e.SDEexplicit_particles}
\d X^i_t &= \Ll( -\frac \al n \sum_{j = 1}^n Y_t^j + X^i_t - \frac{(X^i_t)^3}{3} \Rr) \d t + \sqrt{2\ep} \, \d B_t\, , \\ 
\d Y^i_t &= \Ll( \ \  \frac \al n \sum_{j = 1}^n X_t^j + Y^i_t - \frac{(Y^i_t)^3}{3} \Rr) \d t + \sqrt{2\ep} \, \d B'_t\, ,
\end{align}
via a standard Euler-Maruyama discretization scheme. In the following we took the parameters $N = 500$, $\alpha = 1.5$ and i.i.d. initializations taken as $X_0^i \sim  \mathcal{N}(-0.2,0.25)$ and $Y_0^i \sim  \mathcal{N}(0.4,0.25)$. We simulate until $T = 20$ to observe a steady behavior. We show this for two different diffusion coefficients: $\varepsilon = 0.25$ and $\varepsilon = 0.5$. In the first case the system appears to stabilize on a periodic motion, whereas in the second case, the system seems to converge to a stationary measure which should be, up to particle approximation, the one depicted in~\cite{wang2024open, seo2026local}. 

We show in Figure~\ref{fig:mean_variance} the empirical mean and variance of the particles as a function of time $t$. We observe that in the first case, for which $\varepsilon = 0.25$, these quantities stabilize to seemingly periodic functions, whereas in the second case, they converge to constants.
\begin{figure}
    \centering
    \includegraphics[width = \textwidth]{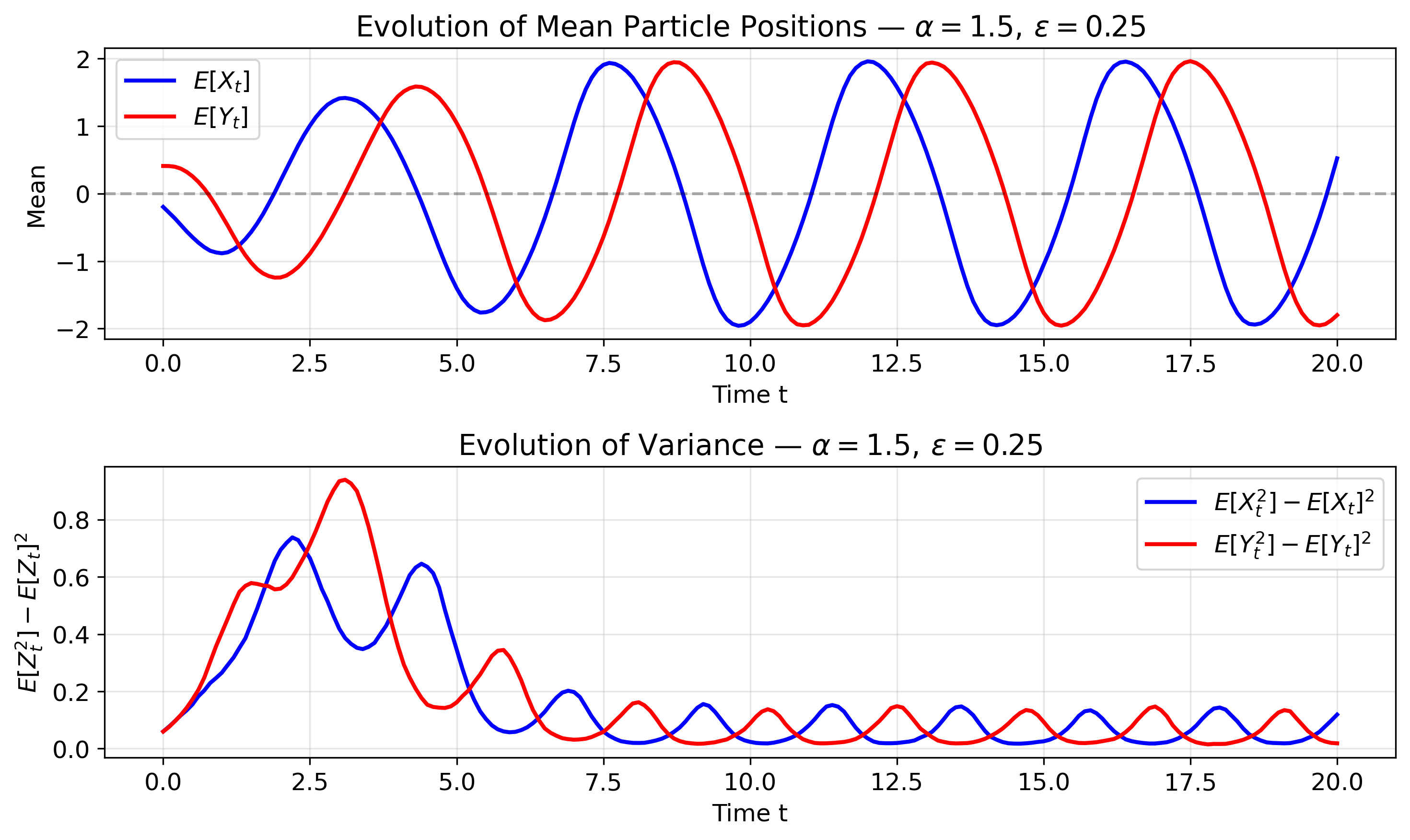}
    \includegraphics[width = \textwidth]{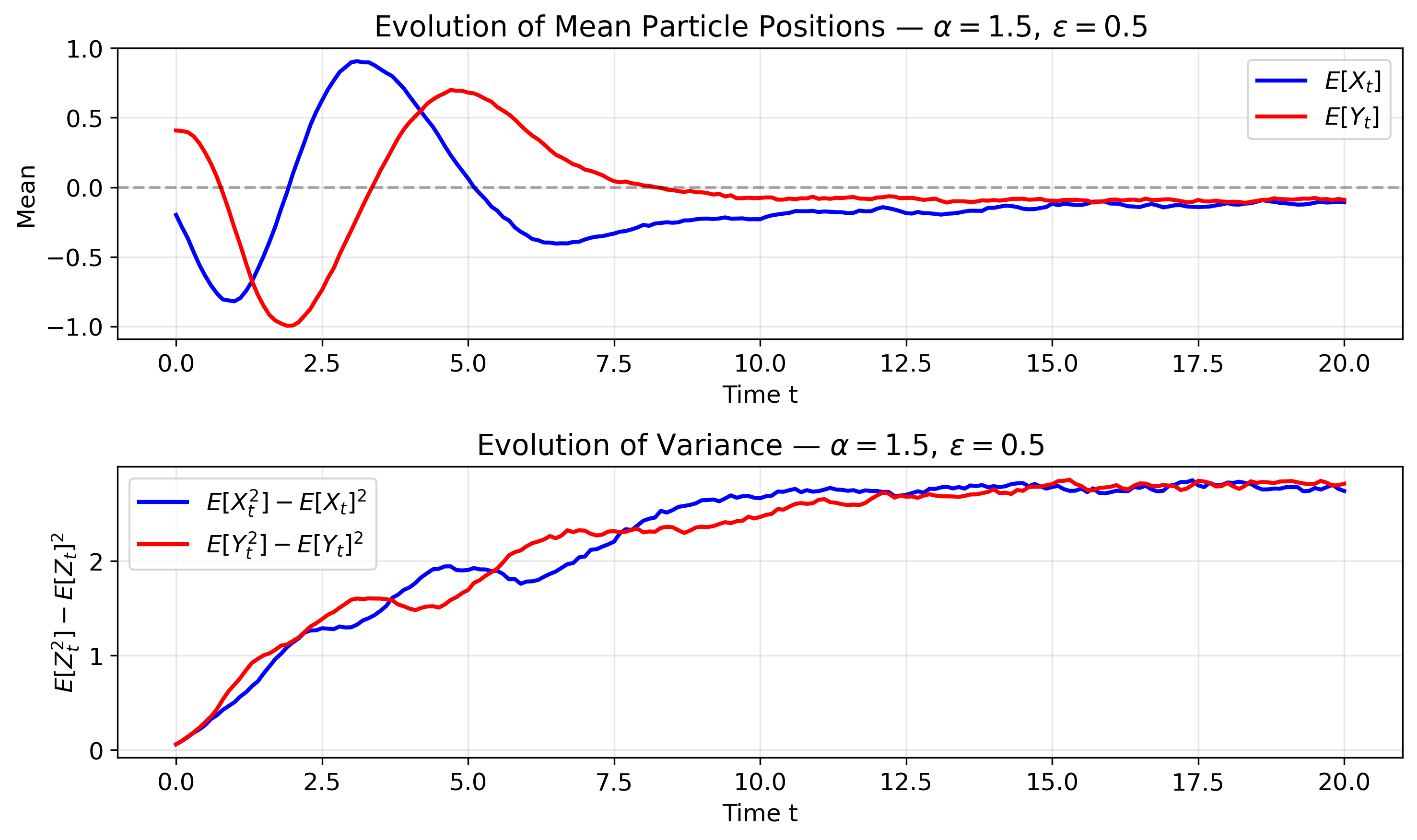}
    \caption{Mean and variance of the particles as a function of time for $\varepsilon = 0.25$ (two top plots) and $\varepsilon = 0.5$ (two bottom plots).}
    \label{fig:mean_variance}
\end{figure}

In Figure~\ref{fig:particles_snapshot}, we show snapshots of the particles at different times for the two different values of $\varepsilon$. We observe that in the first case ($\varepsilon = 0.25$), the particles fluctuate around a periodic cycle, which is close to the limit cycle of the deterministic system (depicted as a black line). On the other hand, in the second case the particles are spread out and seem to be samples of a stationary measure, which resembles a mixture of four Gaussians. Animations of the particle system are available at {\small \url{https://thebiglouloup.github.io/loucaspillaudvivien/minimax_animations.html}}.
\begin{figure}
    \centering
    \begin{subfigure}[b]{0.38\textwidth}
        \centering
        \includegraphics[width=\textwidth]{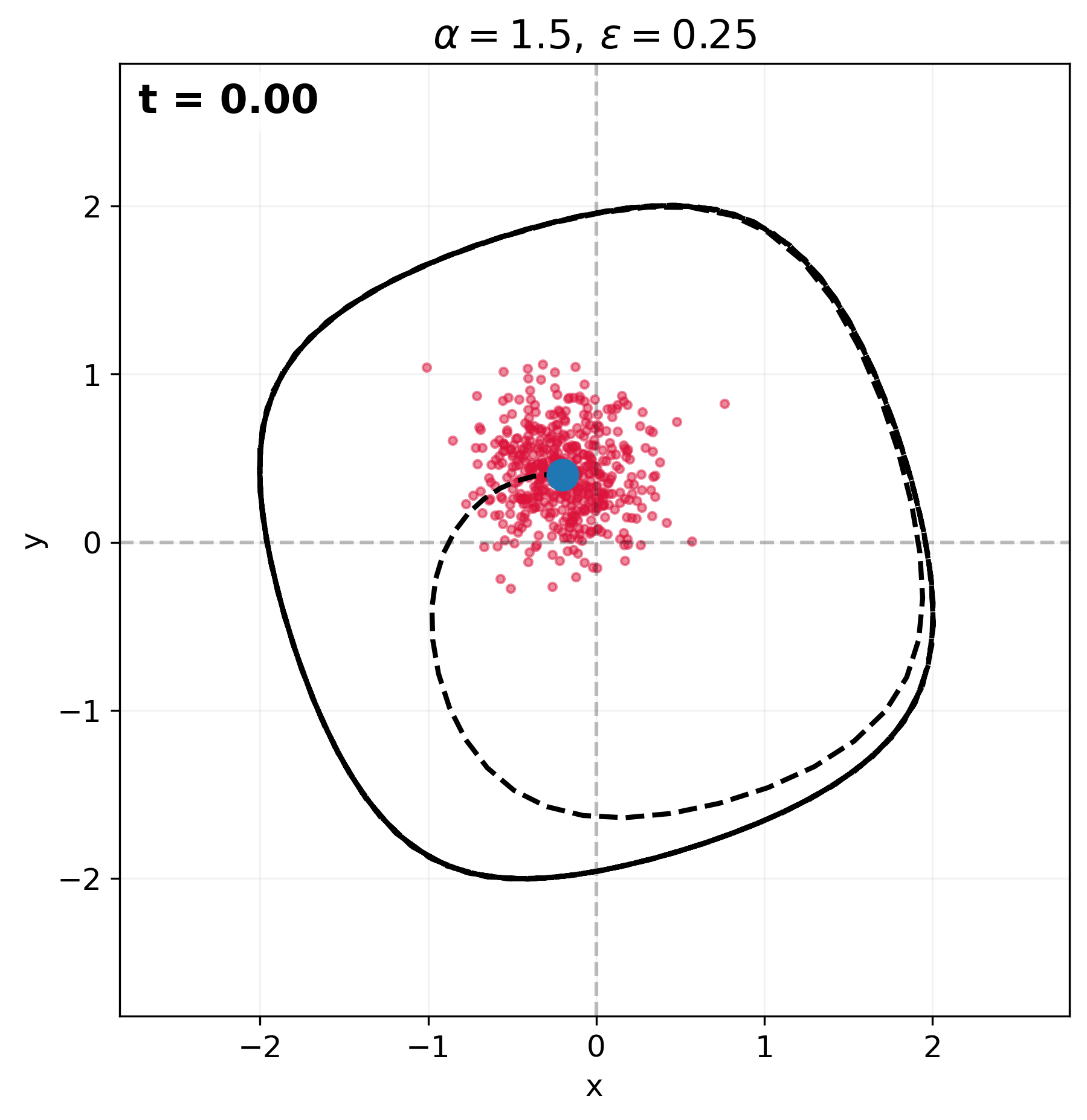}
    \end{subfigure}
    \hfill
    \begin{subfigure}[b]{0.38\textwidth}
        \centering
        \includegraphics[width=\textwidth]{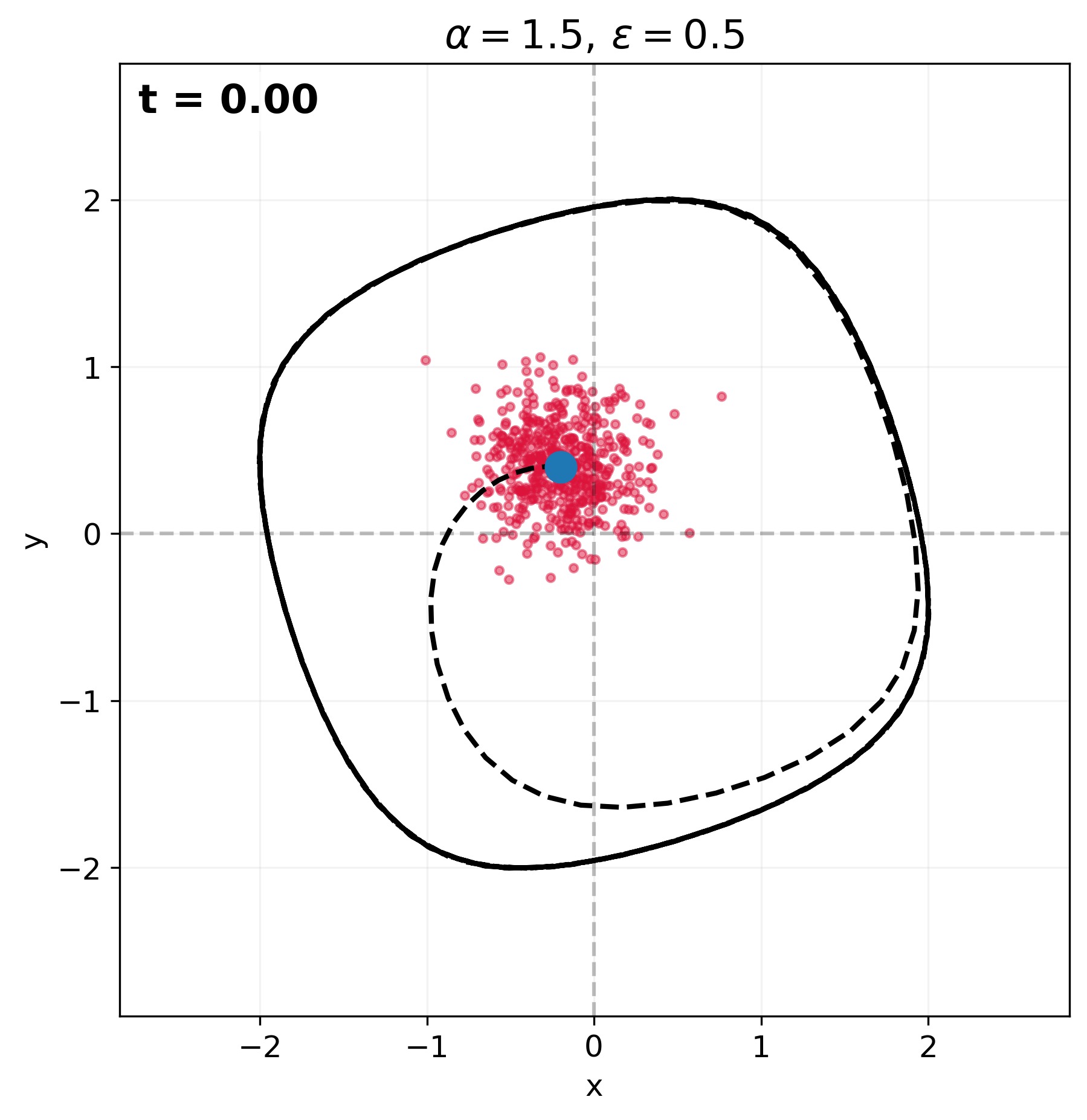}
    \end{subfigure}

    \begin{subfigure}[b]{0.38\textwidth}
        \centering
        \includegraphics[width=\textwidth]{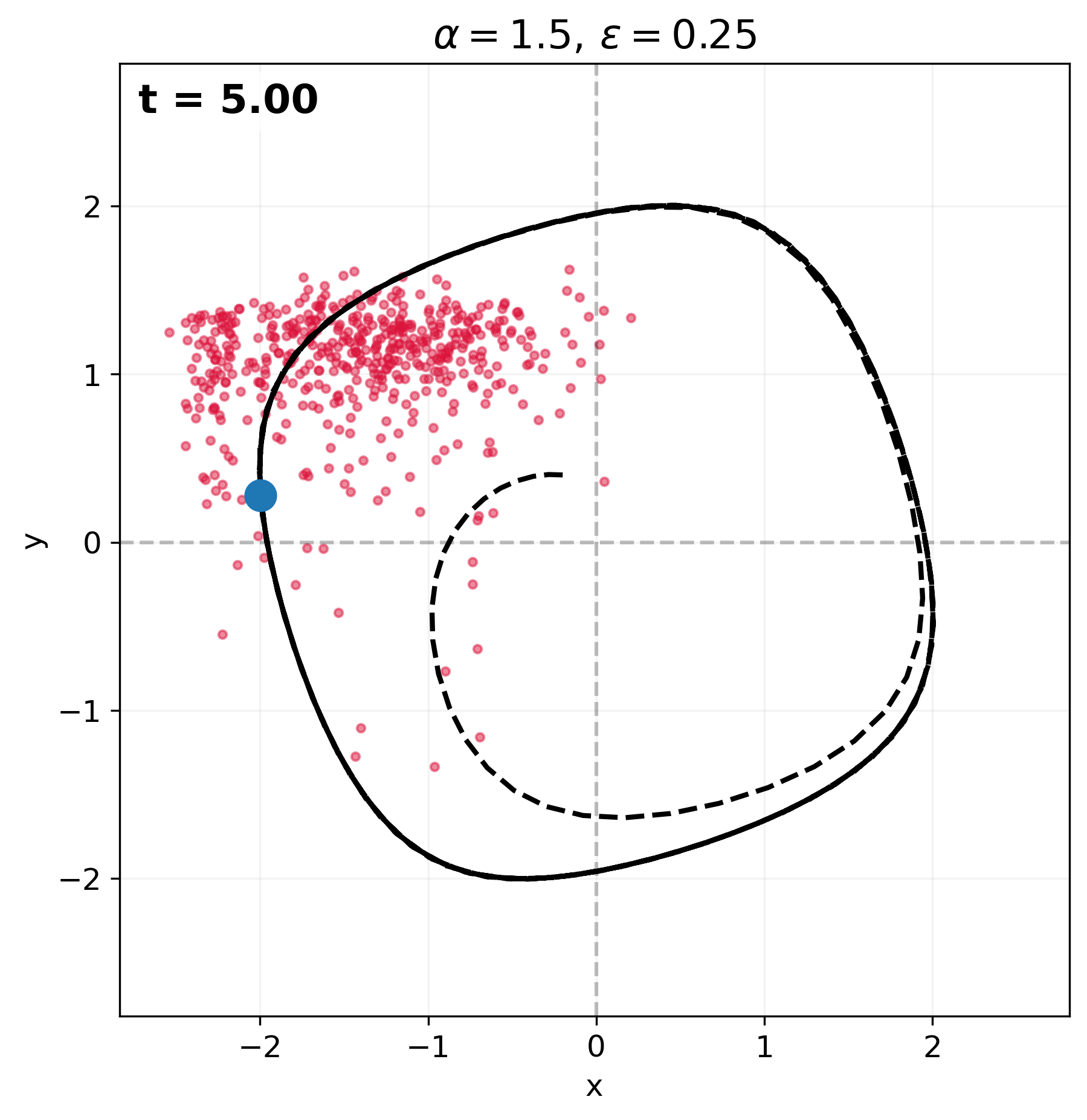}
    \end{subfigure}
    \hfill
    \begin{subfigure}[b]{0.38\textwidth}
        \centering
        \includegraphics[width=\textwidth]{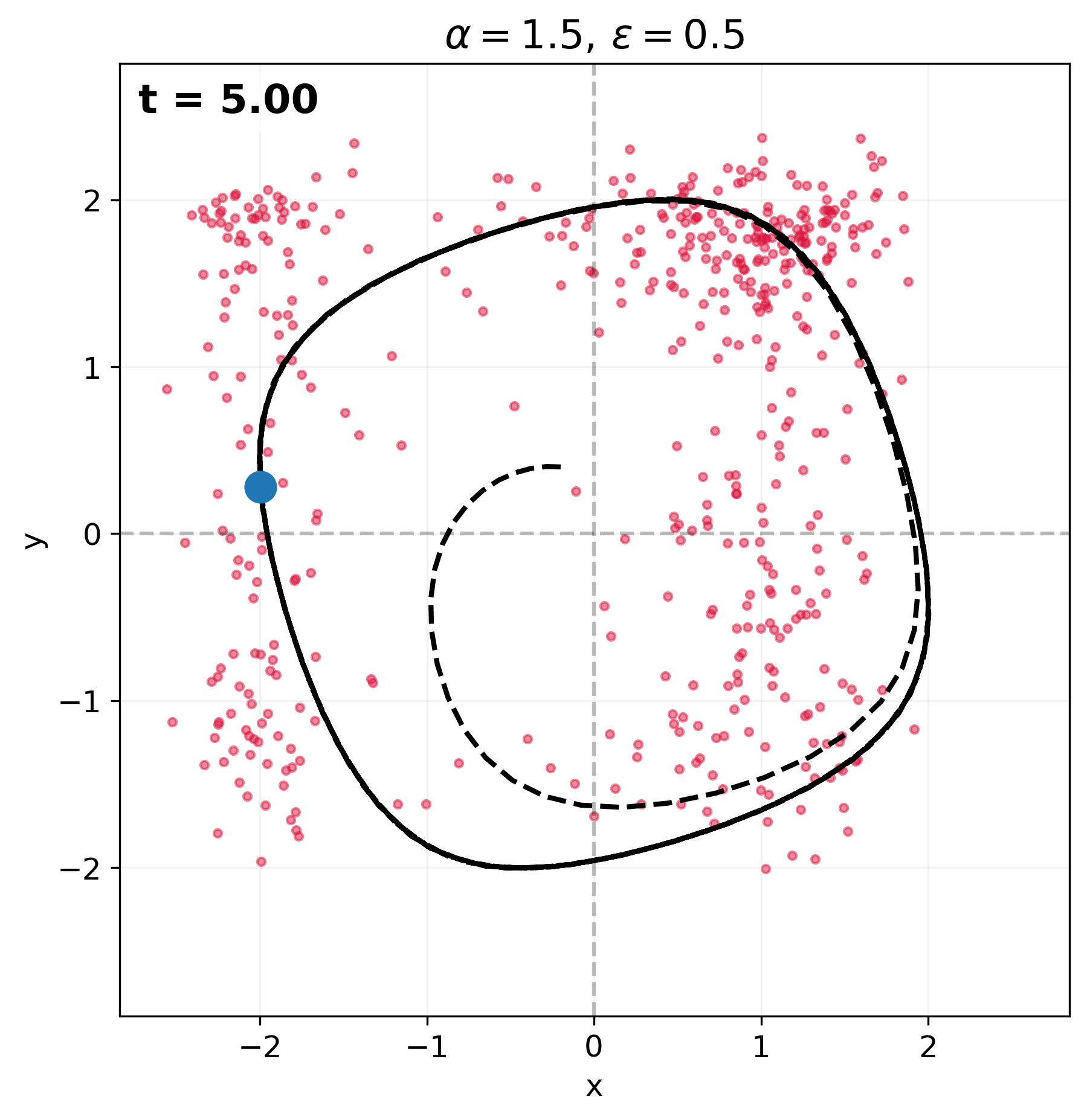}
    \end{subfigure}

        \begin{subfigure}[b]{0.38\textwidth}
        \centering
        \includegraphics[width=\textwidth]{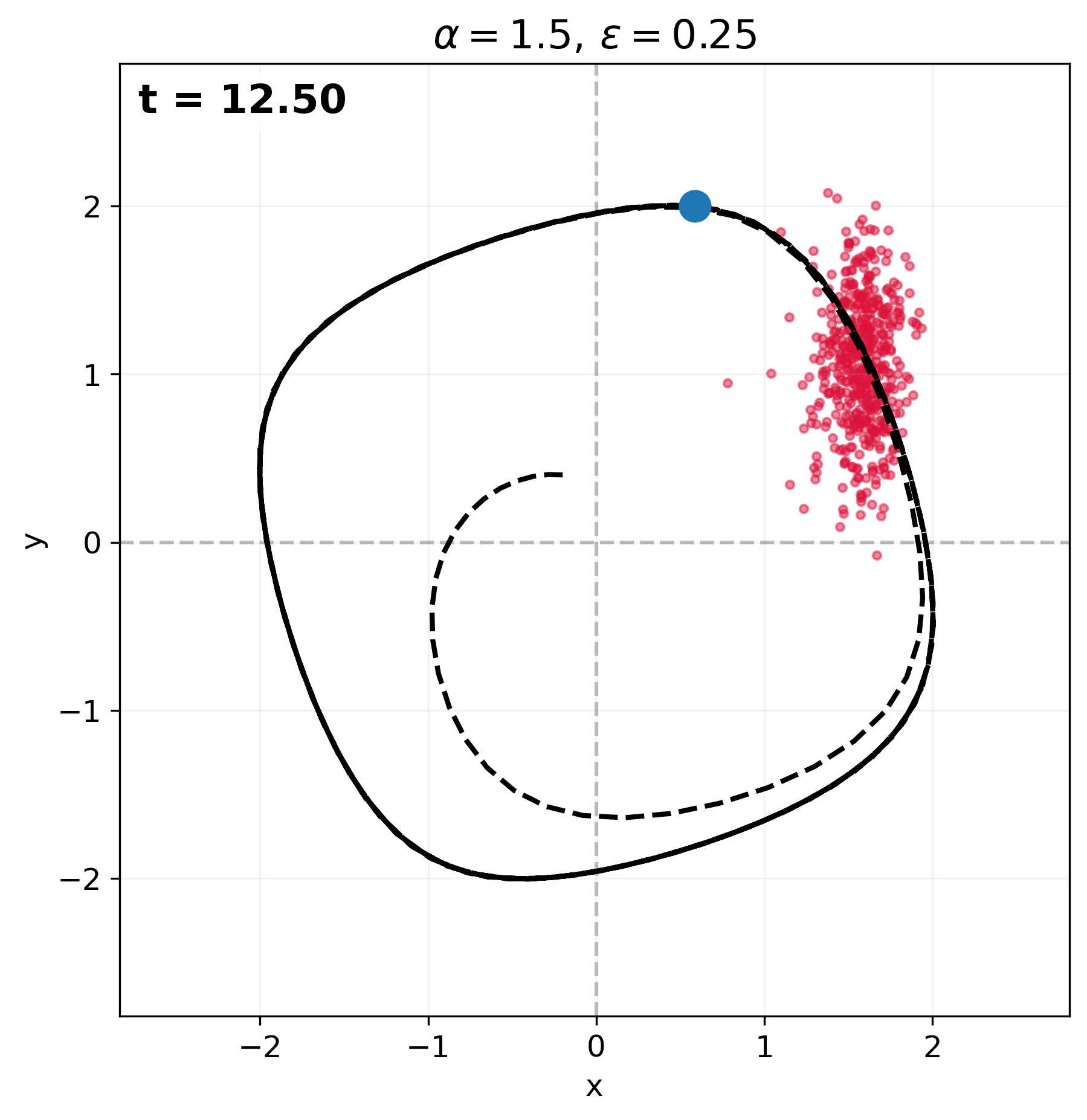}
    \end{subfigure}
    \hfill
    \begin{subfigure}[b]{0.38\textwidth}
        \centering
        \includegraphics[width=\textwidth]{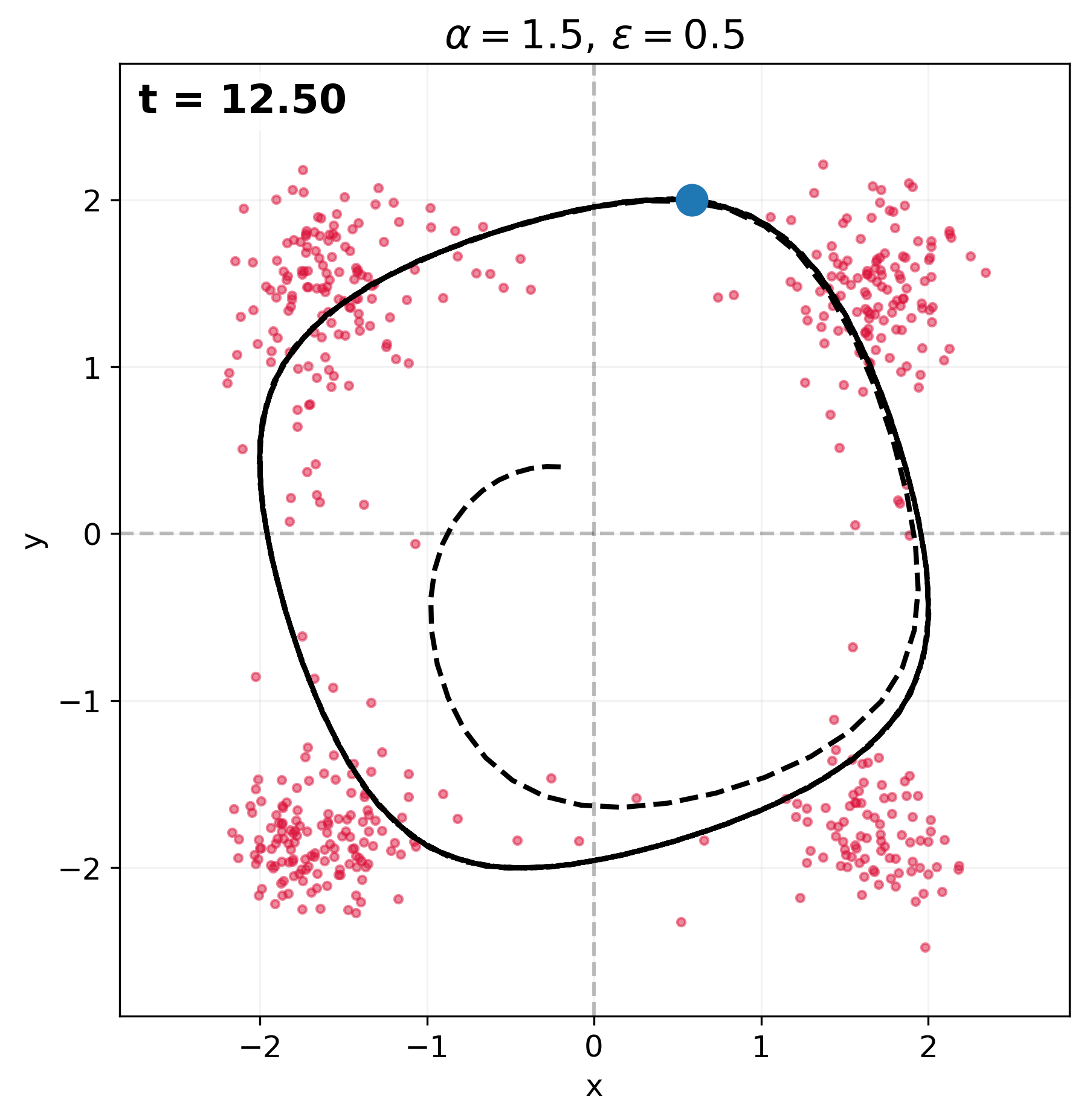}
    \end{subfigure}

    \begin{subfigure}[b]{0.38\textwidth}
        \centering
        \includegraphics[width=\textwidth]{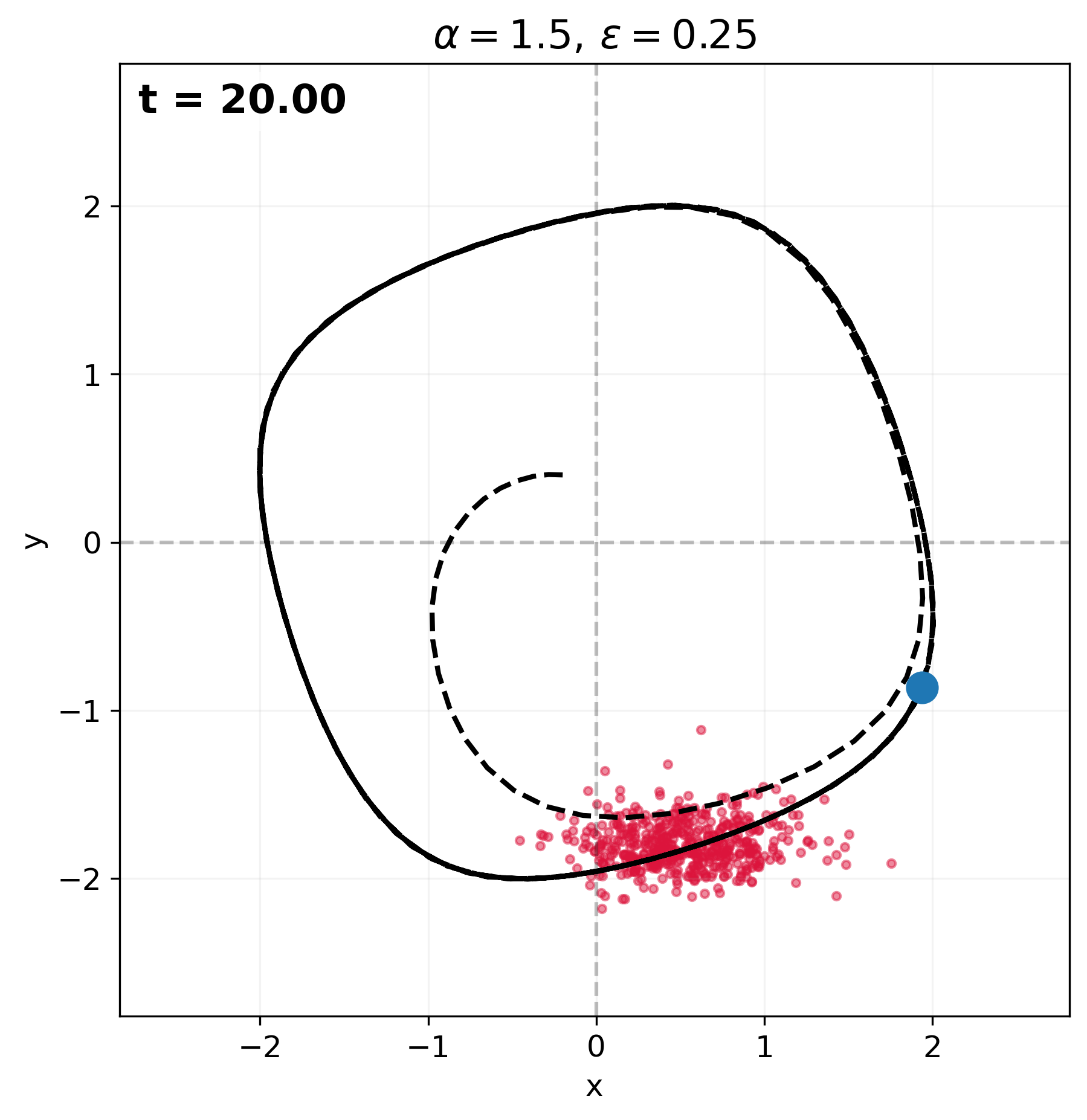}
    \end{subfigure}
    \hfill
    \begin{subfigure}[b]{0.38\textwidth}
        \centering
        \includegraphics[width=\textwidth]{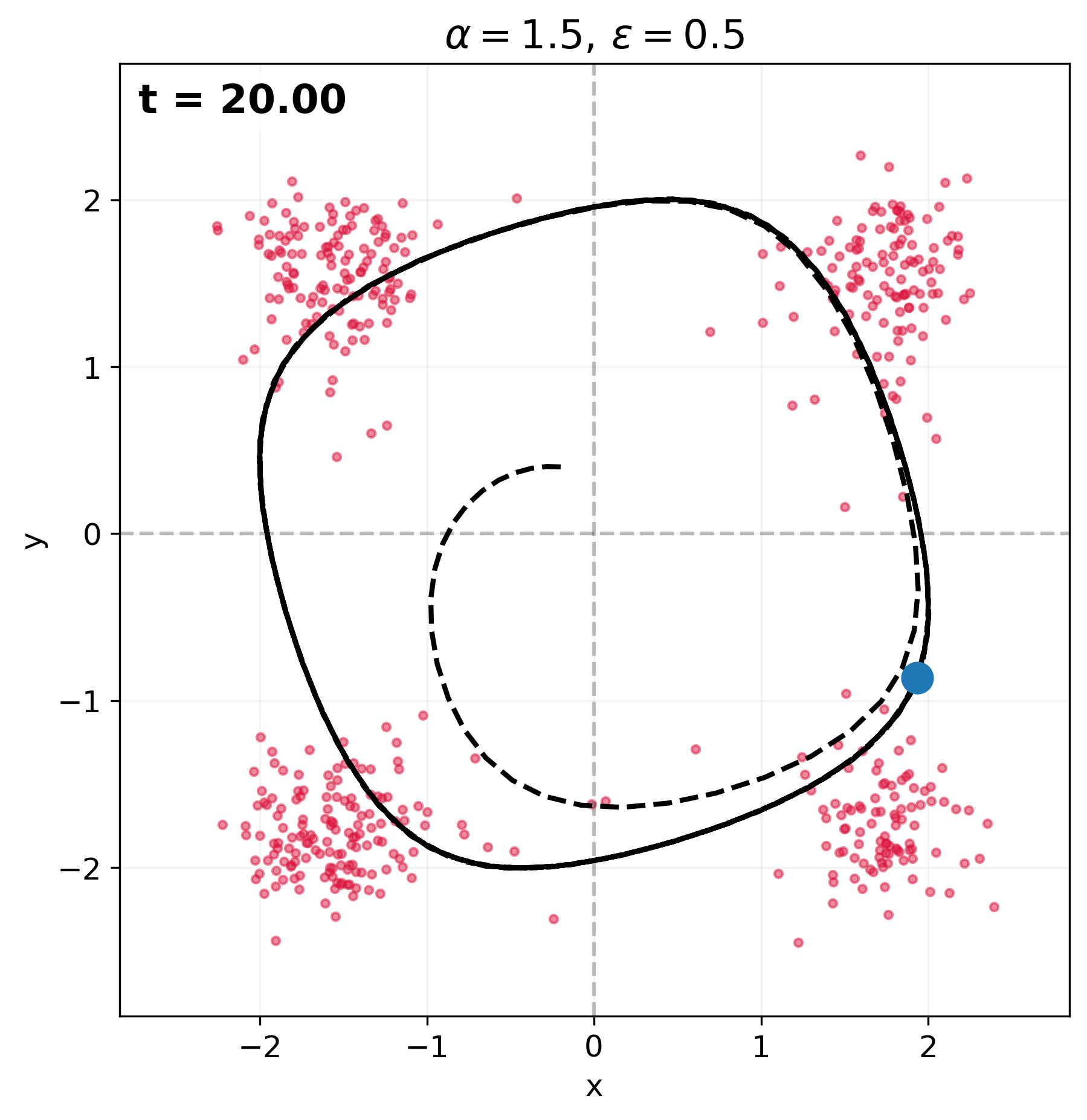}
    \end{subfigure}
    \caption{Particle snapshots for $\varepsilon = 0.25$ (left) and $\varepsilon = 0.5$ (right) at times $t = 0, 5, 12.5, 20$. The red points represent the positions of the $N = 500$ particles, the blue dot stands for the solution of the deterministic trajectory at time $t$. Its trajectory is displayed as a black dotted line, which merges with the limit cycle represented as a continuous black curve.}
    \label{fig:particles_snapshot}
\end{figure}

\clearpage

\medskip

\noindent \textbf{Acknowledgements.} JCM acknowledges the support of the ERC MSCA grant SLOHD (101203974), and of the French National Research Agency (ANR) under the France 2030 grant ANR-24-RRII-0002 operated by the Inria Quadrant Program.

\small
\bibliographystyle{plain}
\bibliography{minmax}

\end{document}